\documentclass[oneside,reqno]{amsart} 

\newcounter{version}
\newcounter{short}
\newcounter{long}
\usepackage{amsmath,amsthm,amscd}
\usepackage{amssymb,amsfonts} 
\input xy
\xyoption{all} \numberwithin{equation}{section}

\theoremstyle{definition}
\newtheorem{dfn}{Definition}[section]
\newtheorem{example}[dfn]{Example}
\newtheorem{rem}[dfn]{Remark}
\theoremstyle{plain}
\newtheorem{thm}[dfn]{Theorem}
\newtheorem{prp}[dfn]{Proposition}
\newtheorem{cor}[dfn]{Corollary}
\newtheorem{lem}[dfn]{Lemma}

\newcommand\hyph{-\penalty0\hskip0pt\relax}

\newcommand\C{{\mathbb C}}
\newcommand\Z{{\mathbb Z}}
\newcommand\N{{\mathbb N}}

\newcommand\T{{\mathbb T}}

\newcommand\LL{{\mathcal L}}

\newcommand\J{{\mathcal J}}
\newcommand\hM{\widehat{M}}
\newcommand\tX{\widetilde{X}}

\newcommand\Ind{{\rm Ind}}
\newcommand\Ann{{\rm Ann}}
\newcommand\Hom{{\rm Hom}}
\newcommand\Span{{\rm Span\,}}
\newcommand\Der{{\rm Der\,}}
\newcommand\gl{gl}
\newcommand\darrow{\longrightarrow {\mkern -27mu} {\raise 6pt \hbox{\it d}} {\mkern 16mu}}
\newcommand\GL{GL_n(\Z)}

\date{}
\begin{document}

\

\title
[Classification of simple $W_n$-modules]
{Classification of simple $W_n$-modules with finite-dimensional weight spaces}
\author{Yuly Billig}
\address{School of Mathematics and Statistics, Carleton University, Ottawa, Canada}
\email{billig@math.carleton.ca}
\author{Vyacheslav Futorny}
\address{ Instituto de Matem\'atica e Estat\'\i stica,
Universidade de S\~ao Paulo,  S\~ao Paulo,
 Brasil}
 \email{futorny@ime.usp.br}
\let\thefootnote\relax

\begin{abstract}
We classify all simple $W_n$-modules with finite-dimensional weight spaces. 
Every such module is either of a highest weight type
or is a quotient of a module of tensor fields on a torus, which was conjectured by Eswara Rao.
This generalizes the classical result of Mathieu on simple weight modules for
the Virasoro algebra. In our proof of the classification
we construct a functor from the category of cuspidal $W_n$-modules to the category
of $W_n$-modules with a compatible action of the algebra of functions on a torus.
We also present a new identity for certain quadratic elements in the universal enveloping algebra
of $W_1$, which provides important information about cuspidal $W_1$-modules.
\end{abstract}

\maketitle

\section{Introduction.}

In 1992 Mathieu \cite{Mat} classified simple modules with finite-dimensional weight spaces for the Virasoro algebra,
proving a conjecture of Kac \cite{Kac}. This was preceded by the work of Kaplansky \cite{Kap} and
Kaplansky-Santharoubane \cite{KS} on classification of simple modules for the Virasoro algebra with $1$-dimensional
weight spaces.

Mathieu proved that simple weight modules fall into two classes: (1) highest/ lowest weight modules and (2) modules
of tensor fields on a circle and their quotients. It was natural to ask the same question about the Lie algebra $W_n$
of vector fields on a torus. In case of $W_n$, however, it was not clear at the time what should be the statement of the
conjecture.

The $W_n$-modules of tensor fields were well-known, in 1970s Gelfand's school studied various cohomological
problems related to these modules \cite{Fu}. However it was not known whether there exist analogues of 
the highest weight modules for $W_n$ with finite-dimensional weight spaces.
It turned out that such modules do exist, and the highest weight type modules for $W_n$ were constructed 
using a theorem of Berman-Billig \cite{BB}, \cite{BZ}. Recently vertex operator realizations of simple
$W_n$-modules of the highest weight type were given in \cite{BF}.

 In 2004 Eswara Rao conjectured that simple modules for $W_n$ with finite{\hyph}dimensional weight spaces also fall into
two classes: (1) modules of the highest weight type and (2) modules of tensor fields on a torus and their quotients.


 Solenoidal Lie algebras (also known as centerless higher rank Virasoro algebras) serve as a bridge
between $W_1$ and $W_n$. A solenoidal Lie algebra is the algebra of vector fields which are collinear at each point 
of the torus to some fixed generic vector. The name is motivated by the fact that the flow lines of such fields are
dense windings on the torus. Solenoidal Lie algebras are very similar to $W_1$ in many regards, yet they have 
the same grading as $W_n$.

Su \cite{Su} classified simple cuspidal modules for the solenoidal Lie algebras (i.e., modules with a uniform bound on dimensions
of weight spaces). Lu-Zhao \cite{LZ} classified all simple modules with finite-dimensional weight 
spaces for this class of algebras.

Using this work, Mazorchuk-Zhao \cite{MZ} described supports (sets of weights) of all simple $W_n$-modules with 
finite-dimensional weight spaces.

In the present paper we complete the classification of such modules, proving Eswara Rao's conjecture.

In order to achieve the result, we need to understand the structure of cuspidal (not necessarily simple) modules for $W_1$,
the solenoidal algebras and $W_n$. For an important subcategory of cuspidal $AW_n$-modules that 
consists of modules admitting the action of the commutative algebra $A$ of functions, this structure is
understood, such $AW_n$-modules were analyzed in \cite{Rao2} and \cite{Jet}. 
Our present approach reduces the classification of simple cuspidal
$W_n$-modules to the classification of simple cuspidal $AW_n$-modules. To achieve this we construct
the functor of coinduction, which associates to a $W_n$-module $M$ a coinduced $AW_n$-module
$\Hom (A, M)$. The coinduced module itself is too big and takes us out of the category of 
modules with finite-dimensional weight spaces. However, we can exploit the fact that $W_n$ is itself
an $A$-module and for each $x \in W_n$, $u\in M$ define $\psi(x,u) \in \Hom (A, M)$ by
$$\psi(x,u) (f) = (fx) u, \quad \text{ for } f\in A.$$
A subspace in the coinduced module spanned by all  $\psi(x,u)$ is an $AW_n$-submodule, which
we call the $A$-cover $\hM$ of $M$. The $A$-cover construction plays a crucial role in our proof.

One of our key results is a theorem that the $A$-cover of a cuspidal module still has finite-dimensional weight
spaces. As we have a homomorphism $\hM \rightarrow M$, $\psi(x, u) \mapsto x u$, which is
surjective under a mild assumption $W_n M = M$, we are able to reduce the classification
of simple cuspidal $W_n$-modules to the classification of simple cuspidal $AW_n$-modules, obtained in
\cite{Rao2}. 

 The proof of the theorem that the $A$-cover of a cuspidal module is again cuspidal, is based on an
interesting new identity for the following quadratic elements in the universal enveloping algebra of $W_1$:
$$\Omega_{k,s}^{(m)} = \sum_{a=0}^m (-1)^a {m \choose a} e_{k-a} e_{s+a} .$$
We prove, in particular, that for $m \geq 2$
\begin{align}
\sum_{i=0}^m \sum_{j=0}^m (-1)^{i+j} & {m \choose i} {m \choose j} 
\left( \left\{ \Omega^{(m)}_{k-i,s-j} , \Omega^{(m)}_{q+i,p+j} \right\}
-       \left\{ \Omega^{(m)}_{k-i,q-j} , \Omega^{(m)}_{s+i,p+j} \right\} \right)
\notag \\
& = (q-s) (p-k+2m) \Omega^{(4m)}_{k+p+2m,s+q-2m},
\label{alg}
\end{align}
where by $\{ X, Y \}$ we denote the anticommutator $XY + YX$.

Using this algebraic identity we show that operators $\Omega_{k,s}^{(m)}$ annihilate any cuspidal
$W_1$-module $M$ for large enough $m$, which depends only on the length of the composition 
series of $M$.

Let us present the classification of simple $W_n$-modules with finite-dimensional weight spaces that we 
establish in this paper. 

 Modules of tensor fields on a torus $T(U, \beta)$ are parametrized by their support $\beta+\Z^n$ and 
a finite-dimensional $\gl_n$-module $U$. The latter is explained by the fact that the local coordinate
changes induce the action of $\gl_n$ on a fiber of a tensor bundle. If a $\gl_n$-module $U$ is simple
the resulting module $T(U,\beta)$ is a simple $W_n$-module, except when it is a module $\Omega^k (\beta)$ 
of differential $k$-forms. The exceptional nature of the modules of differential forms comes from the fact
that they form the de Rham complex, and the differential $d$ is a homomorphism of $W_n$-modules.

Modules of the highest weight type are constructed using the technique of the generalized Verma modules.
We choose a $\Z$-grading on $W_n$, say by degree in $t_n$. Then a degree zero subalgebra is a semidirect
product of $W_{n-1}$ with an abelian ideal. We take a tensor module  for $W_{n-1}$ and let the abelian
ideal act on it via $A$-action rescaled with a complex parameter $\gamma$, yielding a module $T(U, \beta, \gamma)$
for the degree zero subalgebra in $W_n$. We postulate that the positive degree subalgebra acts trivially on
$T(U,\beta, \gamma)$ and construct the induced $W_n$-module $M(U,\beta,\gamma)$. By Theorem of
Berman-Billig \cite{BB}, the simple quotient $L(U,\beta,\gamma)$ has finite-dimensional weight spaces.
Finally, we can twist such modules by a $\GL$-automorphism of $W_n$, which corresponds to a change
in the $\Z$-grading on $W_n$.

The trivial $1$-dimensional $W_n$-module can be constructed as either a quotient of a module of differential forms
or as a module of the highest weight type.

Our main result is the following theorem.
\begin{thm} (cf., Theorem \ref{main})
Every  simple weight $W_n$-module with finite-dimensional weight spaces is isomorphic to one of the
following:

$\bullet$ a module of tensor fields $T(U, \beta)$, where $\beta\in\C^n$ and $U$ is a simple finite-dimensional
$\gl_n$-module, different from an exterior power of the natural $n${\hyph}dimensional $\gl_n$-module,

$\bullet$ a submodule $d \Omega^k (\beta) \subset \Omega^{k+1} (\beta)$ for $0 \leq k < n$ and  $\beta\in\C^n$,

$\bullet$ a module of the highest weight type $L(U,\beta,\gamma)^g$, twisted by $g\in\GL$, where $U$ is a simple
finite-dimensional $\gl_{n-1}$-module, $\beta\in\C^{n-1}$ and $\gamma\in\C$. 
\end{thm}

The structure of the paper is as follows. In Section 2 we discuss Lie algebras of vector fields on a torus,
their solenoidal subalgebras and give basic definitions in their representation theory. In Section 3 we prove our
key identity \eqref{alg} and link it with annihilators of cuspidal modules. In Section 4 we develop
the theory of coinduced modules and study the properties of $A$-covers of cuspidal modules.
Finally, in Section 5 we apply our results to establish the classification of simple $W_n$-modules with
finite-dimensional weight spaces.

\section{Lie algebras of vector fields and their cuspidal modules}

 Consider the Lie algebra $W_n$ of vector fields on an $n$-dimensional torus $\T^n$. The algebra of 
(complex-valued) Fourier polynomials on $\T^n$ is isomorphic to the algebra of Laurent polynomials
$$A = \C [t_1^{\pm 1}, \ldots, t_n^{\pm 1}] ,$$ 
and $W_n$ is the Lie algebra of derivations of $A$. Thus $W_n$ has a natural structure of an $A$-module,
which is free of rank $n$. We choose $d_1 = t_1 \frac{\partial}{\partial t_1}, \ldots, d_n = t_n \frac{\partial}{\partial t_n}$
as a basis of this $A$-module:
$$W_n = \mathop\oplus\limits_{p=1}^n A d_p .$$
Setting the notation $t^r = t_1^{r_1} \ldots t_n^{r_n}$ for
$r = (r_1,\ldots,r_n) \in \Z^n$, we can write the Lie bracket in $W_n$ as follows:
$$[t^r d_a, t^s d_b ] = s_a t^{r+s} d_b - r_b t^{r+s} d_a , \quad a,b = 1,\ldots, n; \ r, s \in \Z^n.$$
 The subspace spanned by $d_1, \ldots, d_n$ is the Cartan subalgebra in $W_n$ and the adjoint action of these
elements induces a $\Z^n$-grading.

In case when $n=1$ we get the Witt algebra $W_1$ with the basis $e_k = t_1^k d_1$, $k\in\Z$, and the bracket
\begin{equation}
[e_r, e_s ] = (s-r) e_{r+s} .
\label{witt}
\end{equation}

One particular family of subalgebras in $W_n$ will play a significant role in this paper. 

\begin{dfn}
We call a vector $\mu\in\C^n$ {\it generic} if  $\mu \cdot r \neq 0$ for all $r \in\Z^n\backslash \{ 0 \}$.
\end{dfn}
\begin{dfn}
Let $\mu$ be a generic vector in $\C^n$ 
and let $d_\mu = \mu_1 d_1 + \ldots + \mu_n d_n$. 
A {\it solenoidal} Lie algebra $W_\mu$ is the subalgebra in $W_n$ which consists of vector fields
collinear to $\mu$ at each point of $\T^n$, $W_\mu = A d_\mu$.
\end{dfn}
The Lie bracket in $W_\mu$ is 
$$ [t^r d_\mu, t^s d_\mu ] = \mu \cdot (s-r) t^{r+s} d_\mu .$$
Denote by $\Gamma_\mu$ the image of $\Z^n$ under the embedding $\Z^n \rightarrow \C$
given by $r \mapsto \mu \cdot r$. Then we can view the solenoidal Lie algebra $W_\mu$ as a
version of the Witt algebra $W_1$ where the indices of the basis elements $e_r$ run not over $\Z$,
but over the lattice $\Gamma_\mu \subset \C$. Here we make the identification $t^r d_\mu = e_{\mu \cdot r}$
and the formula \eqref{witt} remains valid.

The properties of the solenoidal Lie algebras are very similar to those of $W_1$, yet these algebras are close enough
to $W_n$. In particular, $W_n$ can be decomposed (as a vector space) in a direct sum of $n$ solenoidal subalgebras.
This elementary observation will be important for our proof.

 Let us now discuss modules for the Lie algebras of vector fields. A $W_n$-module $M$ is called a {\it weight} module
if $M = \mathop\oplus\limits_{\lambda\in\C^n} M_\lambda$, where the weight space $M_\lambda$ is defined as
$$M_\lambda = \left\{ v \in M \, | \, d_i v = \lambda_i v \hbox{ \ \rm for \ } i = 1, \ldots, n \right\} .$$

In particular, $W_n$ is a weight module over itself and its weight decomposition coincides with its $\Z^n$-grading.

Weight spaces of a $W_\mu$-module are defined as the eigenspaces of $d_\mu$.

Any weight $W_n$-module can be decomposed into a direct sum of submodules corresponding to distinct cosets of $\Z^n$
in $\C^n$. It is then sufficient to study modules supported on a single coset $\beta + \Z^n$, $\beta\in \C^n$.
We are going to impose such an assumption on $M$ for the rest of this paper.

\begin{dfn}
We call $M$ an $AW_n$-module if it is a module for both the Lie algebra $W_n$ and the commutative unital algebra $A$
of functions on a torus, with these two structures being compatible:
\begin{equation}
x (f v) = (x f) v + f (x v), \quad f \in A, x \in W_n, v \in M.
\label{compat}
\end{equation}
\end{dfn} 

For solenoidal Lie algebras the concept of an $AW_\mu$-module is defined in the same way.

If an $AW_n$-module $M$ is a weight module then \eqref{compat} implies that the action of $A$ is compatible 
with the weight grading of $M$: $A_\gamma M_\lambda \subset M_{\lambda+\gamma}$, $\gamma \in \Z^n$,
$\lambda \in \beta + \Z^n$. Suppose an $AW_n$-module $M$ has a weight decomposition with one of the
weight spaces being finite-dimensional. Since all non-zero homogeneous elements of $A$ are invertible, we 
conclude that all weight spaces of $M$ have the same dimension and that $M$ is a free $A$-module of a finite
rank.

\begin{dfn}
A weight module is called {\it cuspidal} if the dimensions of its weight spaces are uniformly bounded by some
constant.
\end{dfn}

Irreducible and indecomposable $AW_n$-modules with finite-dimensional weight spaces are classified in \cite{Rao2}
and \cite{Jet}. Our strategy will be to use this classification in order to describe simple cuspidal $W_n$-modules. 

For the Witt Lie algebra $W_1$ the classification of simple weight modules was obtained by Mathieu \cite{Mat}. Every such module is either 
a highest/lowest weight module,
a module of tensor fields, or a quotient of the module of functions on a circle. Apart from the trivial 1-dimensional module,
highest/lowest weight $W_1$-modules are not cuspidal (\cite{MarPia1}, Corollary III.3), 
while the modules of tensor fields are cuspidal and have weight spaces of dimension at most $1$.

 Let us define $W_1$-modules of tensor fields on a circle.
 For $\alpha, \beta \in \C$ let $T(\alpha, \beta)$ be a vector space with a basis $\{ v_s \, | \, s \in \beta + \Z \}$ and the following action
of $W_1$:
\begin{equation}
e_k v_s = (s + \alpha k) v_{s+k} .
\label{tenden}
\end{equation}

The modules $T(\alpha, \beta)$ are simple unless $\alpha = 0, 1$ and $\beta + \Z = \Z$. The exceptional cases $\alpha = 0, 1$ are
the modules of functions and 1-forms on a circle. The module of functions $T(0,0)$ has a 1-dimensional submodule (constant functions),
and the quotient by this 1-dimensional submodule is a simple $W_1$-module $\overline{T}(0,0)$ with a ``hole'' in a zero weight space. 
Applying the differential map from functions to 1-forms we get that $\overline{T}(0,0)$ is a submodule in $T(1,0)$, with the quotient of
$T(1,0)$ by $\overline{T}(0,0)$ being a 1-dimensional module.

 Simple cuspidal modules for the solenoidal Lie algebra $W_\mu$ have the same description as in the case of $W_1$ \cite{Su}, 
with the modification that the basis vectors $v_s$ are indexed by $s \in \beta+\Gamma_\mu$, and the action of $W_\mu$ 
is still given by the formula \eqref{tenden} with $k \in \Gamma_\mu$. The $W_\mu$-modules $T(\alpha, \beta)$ are 
simple unless $\alpha = 0,1$ and $\beta+\Gamma_\mu = \Gamma_\mu$.

\section{Annihilators of cuspidal modules}

In this section we are going to establish an identity in the universal enveloping algebra of $W_1$
which will be instrumental for the classification of simple cuspidal $W_n$-modules. Using this important identity 
we will be able to show that certain quadratic elements of the universal enveloping algebra 
belong to annihilators of cuspidal modules for $W_1$ and $W_\mu$. Let us mention that annihilators of simple
$W_1$-modules $T(\alpha, \beta)$ were described in \cite{ConMar}. Our results also shed light 
on the structure of extensions of cuspidal $W_1$-modules. Such extensions of
length 2 and 3 were studied in \cite{MarPia2} and \cite{Con}.  

 Let us inductively define {\it differentiators} $\Omega^{(m)}_{k,s} \in U(W_1)$. Set
$$\Omega^{(0)}_{k,s} = e_k e_s$$
and let
$$\Omega^{(m+1)}_{k,s} = \Omega^{(m)}_{k,s} - \Omega^{(m)}_{k-1,s+1} .$$
The element $\Omega^{(m)}_{k,s}$ can be viewed as the $m$-th difference derivative of
$\Omega^{(0)}_{k,s}$. We can express the differentiators in a closed form:
$$\Omega^{(m)}_{k,s} = \sum_{i=0}^m (-1)^i {m \choose i} e_{k-i} e_{s+i} .$$
\begin{example}
$\Omega^{(2)}_{1,-1} = e_1 e_{-1} - 2 e_0^2 + e_{-1} e_1$ is the Casimir element of
$sl_2 = \left< e_{-1}, e_0, e_1 \right>$.
\end{example}
The following Lemma can be easily established by a direct computation:
\begin{lem}
For all $k,s \in \Z$, the differentiators $\Omega^{(3)}_{k,s}$ belong to the annihilators
of all $W_1$-modules of tensor fields $T(\alpha, \beta)$.
\label{basic}
\end{lem}

More generally, if $M$ is a module with a polynomial action of $W_1$ (see \cite{BB} for the exact
definition), then $\Omega^{(m)}_{k,s} \in \Ann(M)$ for large enough $m$.

Now let us state our key identity. We denote by $\left\{ X, Y \right\}$ the anticommutator
$XY + YX$.

\begin{thm} Let $m, r \geq 2$. Then
\begin{align*}
\sum_{i=0}^m \sum_{j=0}^r (-1)^{i+j} {m \choose i} {r \choose j}
\left( \left\{ \Omega^{(m)}_{k-i,s-j} , \Omega^{(r)}_{q+i,p+j} \right\}
-       \left\{ \Omega^{(m)}_{k-i,q-j} , \Omega^{(r)}_{s+i,p+j} \right\} \right) 
\\
= (q-s) \left( (p-k+2r) \Omega^{(2m+2r-1)}_{k+p+2r,s+q-2r}
-  (p-k+2m) \Omega^{(2m+2r-1)}_{k+p+2r-1,s+q-2r+1} \right) .
\end{align*}
\label{magic}
\end{thm}

\begin{cor}
For every $\ell$ there exists $m$ such that for all $k,s$ the differentiator $\Omega^{(m)}_{k,s}$
annihilates every cuspidal $W_1$-module with a composition series of length $\ell$.
\label{ann}
\end{cor}
\noindent
{\it Proof of Corollary \ref{ann}.} Let $M$ be a cuspidal $W_1$-module with a composition series of length $\ell$:
$$ 0 = M_0 \subset M_1 \subset M_2 \subset \ldots \subset M_{\ell} = M$$
with simple quotients $M_{i+1}/M_i$.
We give a proof by induction on $\ell$. Since non-trivial highest/lowest weight $W_1$-modules
are not cuspidal, the composition series of $M$ will have the modules
$T(\alpha,\beta)$, $\overline{T}(0,0)$ and the trivial $1$-dimensional module as possible simple quotients . 
By Lemma \ref{basic}, $\Omega^{(3)}_{k,s}$ belongs 
to the annihilator of every cuspidal module of length $\ell=1$. To prove the step of induction,
assume $\ell > 1$ and write $\ell = \ell_1 + \ell_2$ with $\ell_1, \ell_2 \geq 1$. By induction 
assumption there exist $m, r$ such that $\Omega^{(m)}_{k,s}$, $\Omega^{(r)}_{p,q}$
annihilate all cuspidal modules of lengths $\ell_1, \ell_2$ respectively. Then both
$\Omega^{(m)}_{k,s} \Omega^{(r)}_{p,q}$ and $\Omega^{(r)}_{p,q} \Omega^{(m)}_{k,s}$
annihilate $M$. Indeed, $\Omega^{(r)}_{p,q} M \subset M_{\ell_1}$ since 
$\Omega^{(r)}_{p,q} \in \Ann\left(M/M_{\ell_1}\right)$ and thus
$\Omega^{(m)}_{k,s} \Omega^{(r)}_{p,q} M = 0$. The argument for 
$\Omega^{(r)}_{p,q} \Omega^{(m)}_{k,s}$ is analogous.

If $m=r$ then Theorem \ref{magic} implies that $\Omega^{(4m)}_{k,s} \in \Ann(M)$.
If $m \neq r$ we get that for all $p, k, s, q$ the expression
$$ (p-k+2r) \Omega^{(2m+2r-1)}_{k+p+2r,s+q-2r}
-  (p-k+2m) \Omega^{(2m+2r-1)}_{k+p+2r-1,s+q-2r+1}$$
annihilates $M$. However we can vary the value of $p-k$ while keeping $p+k$ fixed.
This implies that $\Omega^{(2m+2r-1)}_{k+p+2r,s+q-2r} \in \Ann(M)$, which 
completes the proof of the corollary.
\hfill \qed

\ifnum \value{version}=\value{long} 
\begin{example}
Our result implies that $\Omega^{(12)}_{k,s}$ annihilates all cuspidal $W_1$-modules with a composition 
series of length 2. Such modules were classified by Martin-Piard in \cite{MarPia2}.
The answer turns out to be essentially the same as in the case of Lie algebra of vector fields
on a line, which was investigated earlier by Feigin-Fuks \cite{FF}. Feigin-Fuks listed explicit formulas
for such extensions, which are also valid for $W_1$. Here is an extension $M$ from their list:
$$ 0 \rightarrow T\left( \frac{7 - \sqrt{19}}{2}, \beta \right) \rightarrow M \rightarrow 
T\left( \frac{-5 - \sqrt{19}}{2}, \beta \right) \rightarrow 0$$
with basis $\{ u_p, w_p \, | \, p\in \beta+\Z \}$ and the $W_1$-action:
\begin{align*}
&e_k u_p = \left( p + \frac{7 - \sqrt{19}}{2} k \right) u_{p+k} , \\
&e_k w_p = \left( p + \frac{-5 - \sqrt{19}}{2} k \right) w_{p+k}\\
&+ \left( - \frac{22+5 \sqrt{19}}{4} k^7 - \frac{31+7\sqrt{19}}{2} k^6 p
- \frac{25 + 7 \sqrt{19}}{2} k^5 p^2 - 5 k^4 p^3 + 5 k^3 p^4 + 2 k^2 p^5 \right) u_{p+k} .
\end{align*}
We can see that $\Omega^{(0)}_{k,s} w_p = e_k e_s w_p$ has coefficients which are polynomials of
degree $8$ in $k, s$. Thus  $\Omega^{(9)}_{k,s} w_p = 0$. By inspecting the table of possible extensions 
given in \cite{FF} one can see that the lowest value of $m$ for which 
$\Omega^{(m)}_{k,s}$ annihilates all cuspidal modules of length $\ell=2$ is $m=9$. 
\end{example}
\fi
\noindent
{\it Proof of Theorem \ref{magic}.}
\begin{align*}
\sum_{i=0}^m \sum_{j=0}^r (-1)^{i+j} {m \choose i}  & {r \choose j} 
\left( \left\{ \Omega^{(m)}_{k-i,s-j} , \Omega^{(r)}_{q+i,p+j} \right\}
-       \left\{ \Omega^{(m)}_{k-i,q-j} , \Omega^{(r)}_{s+i,p+j} \right\} \right) \\
=  \sum_{i=0}^m \sum_{j=0}^r  \sum_{a=0}^m \sum_{b=0}^r & (-1)^{i+j+a+b}
{m \choose a} {r \choose b} {m \choose i} {r \choose j} \\
\times & \bigg( e_{k-i-a} e_{s-j+a} e_{q+i-b} e_{p+j+b}
+ e_{q+i-b} e_{p+j+b} e_{k-i-a} e_{s-j+a} \\
& \quad- e_{k-i-a} e_{q-j+a} e_{s+i-b} e_{p+j+b}
-  e_{s+i-b} e_{p+j+b} e_{k-i-a} e_{q-j+a} \bigg) . 
\end{align*}
Let us switch indices $a$ with $i$ and $b$ with $j$ in the last two summands. We get
\begin{align*}
 & \sum_{a,i=0}^m \sum_{b,j=0}^r  (-1)^{i+j+a+b}
{m \choose a} {r \choose b} {m \choose i} {r \choose j} \\
& {\hskip 2cm} \times \bigg( e_{k-i-a} e_{s-j+a} e_{q+i-b} e_{p+j+b}
+ e_{q+i-b} e_{p+j+b} e_{k-i-a} e_{s-j+a} \\
& {\hskip 2cm} \quad \quad- e_{k-i-a} e_{q+i-b} e_{s-j+a} e_{p+j+b}
-  e_{s-j+a} e_{p+j+b}e_{k-i-a} e_{q+i-b} \bigg) 
\end{align*}
\begin{align*}
& =  \sum_{a,i=0}^m \sum_{b,j=0}^r  (-1)^{i+j+a+b}
{m \choose a} {r \choose b} {m \choose i} {r \choose j} \\
& \times \bigg( e_{k-i-a} [e_{s-j+a},  e_{q+i-b}] e_{p+j+b}
+ e_{p+j+b}  e_{k-i-a} [ e_{q+i-b}, e_{s-j+a}] \\
& + [e_{q+i-b}, e_{p+j+b}] e_{k-i-a}  e_{s-j+a} 
  + e_{p+j+b} [e_{q+i-b},  e_{k-i-a}]  e_{s-j+a} \\ 
& + e_{p+j+b} [e_{k-i-a}, e_{s-j+a}] e_{q+i-b}
+ [e_{p+j+b}, e_{s-j+a}] e_{k-i-a} e_{q+i-b} \bigg) .
\end{align*}
Here the part involving the last $4$ summands vanishes. The argument in each case
is the same, so we will show this for the last term, which equals:
$$ \sum_{a,i=0}^m \sum_{b,j=0}^r  (-1)^{i+j+a+b}
{m \choose a} {r \choose b} {m \choose i} {r \choose j}
(s-p-2j+a-b) e_{s+p+a+b}  e_{k-i-a} e_{q+i-b} .$$
We can see that the resulting monomials are independent of $j$. Taking the sum
in $j$ we get zero, since for $r \geq 2$
$$\sum_{j=0}^r (-1)^j {r \choose j} = \sum_{j=0}^r (-1)^j j {r \choose j} = 0 .$$
The sum with the first 2 summands may be expressed in the following way:
\begin{align*}
& \sum_{a,i=0}^m \sum_{b,j=0}^r  (-1)^{i+j+a+b}
{m \choose a} {r \choose b} {m \choose i} {r \choose j} \\
& \times \bigg( [e_{k-i-a},  e_{p+j+b}] [e_{s-j+a},  e_{q+i-b}]
+ e_{k-i-a} [ [e_{s-j+a},  e_{q+i-b}], e_{p+j+b}]   \bigg) .
\end{align*}
The sum involving the last term again vanishes since the monomial in the resulting
expression
$$(q-s+i+j-a-b)(p-s-q-i+2j-a+2b) e_{k-i-a} e_{p+s+q+i+a}$$
is independent of $j$ and $b$ and evaluating the sum in both $j$ and $b$ we get
zero since every term in the expansion of the coefficient vanishes under summation in one 
of these two indices. Thus we are left with
\begin{align*}
 \sum_{a,i=0}^m & \sum_{b,j=0}^r  (-1)^{i+j+a+b}
{m \choose a} {r \choose b} {m \choose i} {r \choose j} \\
&\times (q-s+i+j-a-b)(p-k+i+j+a+b) e_{k+p-i+j-a+b} e_{s+q+i-j+a-b} .
\end{align*}
In the above expression everything except the factor $q-s+i+j-a-b$ is symmetric
in $\{ i, a\}$ and $\{j, b\}$. We can symmetrize in these pairs of variables and get
\begin{align*}
 (q-s) \sum_{a,i=0}^m & \sum_{b,j=0}^r  (-1)^{i+j+a+b}
{m \choose a} {r \choose b} {m \choose i} {r \choose j} \\
&\times (p-k+(i+a)+(j+b)) e_{k+p-i+j-a+b} e_{s+q+i-j+a-b} .
\end{align*}
Make a change of variables replacing $j$ with $r-j$ and $b$ with $r-b$:
\begin{align*} 
(q-s) \sum_{a,i=0}^m & \sum_{b,j=0}^r (-1)^{i+j+a+b}
{m \choose a} {r \choose b} {m \choose i} {r \choose j} \\
& \times (p-k+2r+(i+a)-(j+b)) e_{k+p+2r-i-j-a-b} e_{s+q-2r+i+j+a+b} .
\end{align*}
Set $u = i+j+a+b$, which ranges from $0$ to $2m+2r$. Then the coefficient
at $e_{k+p+2r-u} e_{s+q-2r+u}$ is the same as the coefficient at $t^u$ in
\begin{align*}
 (q-s) & \sum_{a,i=0}^m \sum_{b,j=0}^r  (-1)^{i+j+a+b}
{m \choose a} {r \choose b} {m \choose i} {r \choose j} 
(p-k+2r+(i+a)-(j+b)) t^{i+j+a+b} \\
 = & (q-s) \big( (p-k+2r) (1-t)^m (1-t)^m (1-t)^r (1-t)^r \\
& + 2 (1-t)^{2r+m} t \frac{d}{dt} (1-t)^m
- 2 (1-t)^{2m+r} t \frac{d}{dt} (1-t)^r \big) \\
 = & (q-s) \left( (p-k+2r) (1-t)^{2m+2r}
- 2m t (1-t)^{2m+2r-1} + 2r  t (1-t)^{2m+2r-1} \right) \\
 = & (q-s) \left( (p-k+2r) (1-t)^{2m+2r-1} -  (p-k+2m) t (1-t)^{2m+2r-1} \right) .
\end{align*}
The coefficient at $t^u$ is
$$(q-s)  \left( (p-k+2r) (-1)^u {2m+2r-1 \choose u} -  (p-k+2m) (-1)^{u-1}{2m+2r-1 \choose u-1} \right),$$
which is the same as the coefficient at $e_{k+p+2r-u} e_{s+q-2r+u}$ in
$$ (q-s) \left( (p-k+2r) \Omega^{(2m+2r-1)}_{k+p+2r,s+q-2r}
-  (p-k+2m) \Omega^{(2m+2r-1)}_{k+p+2r-1,s+q-2r+1} \right), $$
which completes the proof of the Theorem.
\hfill \qed

\begin{rem}
In the proofs of Theorem \ref{magic} and Corollary \ref{ann} we relied little
on the specific expressions for the structure constants for the Lie bracket in $W_1$ and its action on
simple cuspidal modules. What matters is the fact that the structure constants 
are polynomial. Thus our methods are applicable to a wider class of Lie algebras with
polynomial multiplication and their modules with polynomial action, which  were introduced in 
\cite{BB}.
\end{rem}
\ifnum \value{version}=\value{long} 
\begin{rem}
We can see from Theorem \ref{magic} that the subspace in $U(W_1)$ spanned by the anticommutators
of differentiators contains differentiators of higher order. Rather surprisingly, it appears that the subspace
spanned by the commutators of the differentiators does not contain any non-zero quadratic elements.
\end{rem}
\fi
Let us now state the analogue of Theorem \ref{magic} for the solenoidal Lie algebra
$W_\mu$. For $k, s, h \in \Gamma_\mu$, $m \geq 0$, define the differentiators in $U(W_\mu)$:
$$\Omega^{(m,h)}_{k,s} = \sum_{i=0}^m (-1)^i {m \choose i} e_{k-ih} e_{s+ih} .$$
We have the following
\begin{thm}
Let $m, r \geq 2$, $k,s,p,q,h \in \Gamma_\mu \subset \C$. Then
\begin{align*}
\sum_{i=0}^m \sum_{j=0}^r (-1)^{i+j} {m \choose i} {r \choose j}
\left( \left\{ \Omega^{(m,h)}_{k-ih,s-jh} , \Omega^{(r,h)}_{q+ih,p+jh} \right\}
-       \left\{ \Omega^{(m,h)}_{k-ih,q-jh} , \Omega^{(r,h)}_{s+ih,p+jh} \right\} \right)  \\
= (q-s) \left( (p-k+2rh) \Omega^{(2m+2r-1,h)}_{k+p+2rh,s+q-2rh}
-  (p-k+2mh) \Omega^{(2m+2r-1,h)}_{k+p+(2r-1)h,s+q-(2r-1)h} \right) .
\end{align*}
\label{solfor}
\end{thm}
\begin{cor}
For every $\ell$ there exists $m$ such that for all $k, s, h \in \Gamma_\mu$ the
differentiator $\Omega^{(m,h)}_{k,s}$ annihilates every cuspidal $W_\mu$-module
with a composition series of length $\ell$.
\label{solann}
\end{cor}
The proofs of Theorem \ref{solfor} and Corollary \ref{solann} are exactly the same as 
in the case of $W_1$.

\section{Coinduced modules}
In this section $\LL$ will denote either $W_n$ or $W_\mu$. The goal of this section is to prove 
the following
\begin{thm}
Let $M$ be a cuspidal $\LL$-module satisfying $\LL M = M$. Then there exist a cuspidal 
$A\LL$-module $\hM$ and a surjective homomorphism of $\LL$-modules:
$$\hM \rightarrow M .$$
\label{proj}
\end{thm}
Our main tool will be the coinduction functor.  Let $M$ be an $\LL$-module.
\begin{dfn}
A module {\it coinduced} from $M$ is the space $\Hom (A, M)$ with the following action of $\LL$ and $A$:
\begin{align}
(x \phi) (f) &= x (\phi(f)) - \phi(x(f)) , 
\label{coindL}\\
(g \phi) (f) &= \phi(gf), \quad \text{ for } \phi \in \Hom(A, M), \; x \in \LL, \; f,g \in A.
\label{coindA}
\end{align}
\end{dfn}
\begin{prp}
(1) The coinduced module $\Hom(A, M)$ is an $A\LL$-module.

(2) The map
\begin{align*}
\pi: \Hom(A,M) &\rightarrow M, \\ 
\phi &\mapsto \phi(1),
\end{align*}
is a surjective homomorphism of $\LL$-modules.
\end{prp}
\begin{proof}
Let us first verify that \eqref{coindL} defines an $\LL$-module structure on
$\Hom(A, M)$:
\begin{align*}
& (x (y \phi))(f) - (y(x\phi))(f) 
= x((y \phi) (f)) - (y \phi) (xf) - y ((x \phi)(f)) + (x\phi)(yf) \\
& = x(y(\phi(f))) - x(\phi(yf)) - y(\phi(xf)) + \phi(y(xf)) \\
& -y (x(\phi(f))) + y(\phi(xf)) + x(\phi(yf)) - \phi(x(yf)) \\
&  = [x,y] (\phi(f)) - \phi([x,y] f) 
 = ([x,y] \phi)(f), \quad \text{ for } \phi \in \Hom(A,M), \; x,y \in\LL, \; f\in A . 
\end{align*}
Clearly, \eqref{coindA} defines an $A$-module structure on $\Hom(A, M)$. To show that
the $\LL$-module and the $A$-module structures are compatible, we need to check that
$$x (g \phi) = x(g) \phi + g (x \phi) .$$
Indeed,
\begin{align*}
(x(g) \phi) (f) + (g x(\phi))(f) = \phi(x(g) f) + x(\phi) (gf) 
= \phi(x(g) f) +  x(\phi(gf)) - \phi(x (gf)) \\
= \phi(x(g) f) + x(\phi(gf)) - \phi(x(g) f) - \phi(g x (f)) = x((g\phi)f) - (g \phi)(x f) 
=(x (g\phi))(f) .
\end{align*}
In order to establish the claim of part (2), we need to show that $x \phi$ is mapped to 
$x (\phi(1))$, i.e., $(x \phi)(1) = x (\phi(1))$, which is true since $x(1) = 0$. 
The surjectivity of $\pi$ is obvious.
\end{proof}

The shortcoming of the coinduced module $\Hom(A,M)$ is that it is too big. We are now
going to construct a smaller submodule in $\Hom(A,M)$ which will serve well for our purposes.
\begin{dfn}
An $A$-{\it cover} of an $\LL$-module $M$ is the subspace $\hM \subset \Hom(A,M)$ spanned
by the set $\{ \psi(x,u) \; | \;  x \in\LL, u \in M \}$, where $\psi (x, u) \in \Hom(A, M)$ is given by
$$\psi(x, u) (f) = (f x) u, \quad \text{for } f \in A.$$
\end{dfn}
\begin{prp}
(1) The action of $\LL$ and $A$ on $\hM$ is the following:
\begin{align}
y \psi(x, u) &= \psi([y,x], u) + \psi(x, yu), 
\label{covL} \\
g \psi(x, u) &= \psi(gx, u), \quad \text{for } x, y \in \LL, \; u \in M,\; g \in A.
\label{covA}
\end{align}

(2) If $M$ is a weight module then so is $\hM$.

(3) $\pi(\hM) = \LL M$.
\label{Astr}
\end{prp}
\begin{proof}
 In proving (1) we are going to use the fact that the adjoint representation of $\LL$
is an $A\LL$-module:
$$[y, fx] = y(f) x + f [y, x] .$$
Then
\begin{align*}
&(y \psi(x,u))(f) = y( \psi(x,u)(f)) - \psi(x,u)(y f)
= y((fx)u) - ((yf) x) u \\
&= [y, fx] u + (f x) (y u) - ((yf) x) u 
= ((y f) x ) u + (f [y, x]) u + (f x) (y u) - ((y f) x) u \\
&= \psi( [y, x], u)(f) + \psi(x, yu) (f) .
\end{align*}
For the action of $A$ we have 
$$(g \psi(x,u))(f) = \psi(x,u) (gf) = ((fg) x) u = \psi(gx, u)(f) .$$
Let us prove the second part of the proposition. If $M$ is a weight module
then $\hM$ is spanned by $\psi(x, u)$ with $x$ and $u$ being eigenvectors with respect to the action
of the Cartan subalgebra of $\LL$. Then \eqref{covL} shows that $\psi(x,u)$ is also an eigenvector 
for the action of the Cartan subalgebra. Since $\hM$ is spanned by its weight vectors, it is a weight module.

Finally, for part (3) we note that $\pi (\psi(x, u)) = \psi(x, u)(1) = x u$, which implies the claim.
\end{proof}

\begin{rem}
It follows from \eqref{covL} that the $A$-cover $\hM$ can also be constructed as a quotient of $\LL \otimes M$
(where the first tensor factor is the adjoint module) by a submodule
$$ \left\{ \sum_i x_i \otimes u_i \; \big| \; \sum_i (f x_i) u_i = 0 \quad\text{for all } f \in A \right\} .$$
\end{rem}

\ifnum \value{version}=\value{long} 
\begin{example} 
One of the technical difficulties in working with $W_n$-modules is the existence of modules with
``holes'' and ``bumps'' at the zero weight space. Let us show how the use of $A$-cover remedies
this complication. Consider a $W_1$-module $M = \overline{T}(0,0) = \C [t, t^{-1}] / \left< 1 \right>$.
We fix a spanning set $u_j = t^j$ with a provision $u_0 = 0$. The action of $W_1$ is 
$e_k u_s = s u_{s+k}$. Let us construct the $A$-cover of $M$:
$$\psi(e_k, u_s) (t^m) = (t^m e_k) u_s = e_{k+m} u_s = s u_{s+k+m} .$$
Introduce $\theta_j \in \Hom(A, M)$, given by
$$\theta_j (t^m) = u_{m+j} .$$
Then we can see that $\psi(e_k, u_s) = s \theta_{k+s}$. Thus $\hM$ has a basis
$\{ \theta_j | j \in \Z \}$ with the $A W_1$-action $e_p \theta_j = j \theta_{j+p}$, $t^p \theta_j = \theta_{j+p}$.
Note that $\theta_0 \neq 0$ and $\hM \cong \C [t, t^{-1}]$. We see that taking the $A$-cover fills in the hole in $M$.
\end{example}
\begin{example}
Let $M$ be the adjoint representation for the Virasoro algebra, viewed as a $W_1$-module. The $W_1$-action
can be written as
\begin{align*}
e_k u_j & = (j-k) u_{j+k} + \delta_{k+j, 0} k^3 z, \\
e_k z & = 0.
\end{align*}
Then $\psi(e_k, u_j)(t^m) = e_{k+m} u_j = (j-k-m) u_{j+k+m} + \delta_{k+m+j,0} (-j)^3 z$,
while $\psi(e_k, z) = 0$.
Consider the following elements in $\Hom(A, M)$:
\begin{align*}
\tau_j (t^m) &= (j+m) u_{j+m}, \\
\theta_j (t^m) &= u_{j+m}, \\
\eta_j (t^m) &= \delta_{j+m,0} z .
\end{align*}
Then $\psi(e_k, u_j) = - \tau_{k+j} + 2 j \theta_{k+j} - j^3 \eta_{k+j}$. 
One can check that the $A$-action is $t^p \tau_j = \tau_{j+p}$,  $t^p \theta_j = \theta_{j+p}$, $t^p \eta_j = \eta_{j+p}$,  
while $W_1$ acts in the following way:
\begin{align*}
e_p \tau_j &= (j - 2p) \tau_{j+p} + 2 p^2 \theta_{j+p} - p^4 \eta_{j+p}, \\
e_p \theta_j &= (j-p) \theta_{j+p} + p^3 \eta_{j+p}, \\
e_p \eta_j &= (j+p) \eta_{j+p} .
\end{align*}
In this case $\hM$ is a weight module with $3$-dimensional weight spaces. The map $\pi: \hM \rightarrow M$
is given by $\pi(\tau_j) = j u_j$, $\pi(\theta_j) = u_j$, $\pi(\eta_j) = \delta_{j,0} z$.
\end{example}
\fi

Assume now that $M$ is a cuspidal $\LL$-module. Its $A$-cover $\hM$ is a weight module, however it is not
a priori clear that the weight spaces of $\hM$ are finite-dimensional. 
\ifnum \value{version}=\value{long} 
In the above examples this was the case due to the fact that the structure constants in the modules were polynomial.
\fi
We are going to prove next that $\hM$ is indeed cuspidal. The key ingredient in the proof is Corollary \ref{solann}.
We begin with the solenoidal Lie algebra.
\begin{thm}
Let $M$ be a cuspidal module for the solenoidal Lie algebra $W_\mu$. Then the $A$-cover $\hM$ of $M$ is
cuspidal.
\label{solcusp}
\end{thm}
\begin{proof}
Without loss of generality let us assume that $\beta = 0$ in case when $\beta+\Gamma_\mu = \Gamma_\mu$.
This ensures that $e_0$ acts with a non-zero scalar on every weight space of $M$ different from $M_\beta$.

For $k \in \Gamma_\mu$, $\lambda \in \beta+ \Gamma_\mu$ denote by $\psi(e_k, M_\lambda)$ a subspace
$$\psi(e_k, M_\lambda) = \left\{ \psi(e_k, u) \; \big| \; u \in M_\lambda \right\} \subset \hM .$$
Since weight spaces of $M$ are finite dimensional, so are the spaces $\psi(e_k, M_\lambda)$.

Since $\hM$ is an $A$-module, it is sufficient to show that one of its weight spaces is finite-dimensional, then all other weight
spaces will have the same dimension. Fix a weight $\beta + s$, $s \in \Gamma_\mu$ and let us prove that $\hM_{\beta+s}$
is finite-dimensional. The space $\hM_{\beta+s}$ is spanned by the set
$$ \mathop\cup\limits_{k \in \Gamma_\mu} \psi(e_{s-k}, M_{\beta+k}) .$$
Choose a basis $\left\{ \delta_1, \ldots, \delta_n \right\}$ of the lattice $\Gamma_\mu$ and introduce a norm on $\Gamma_\mu$:
$$\| k \| = \sum\limits_{j=1}^n |k_j|, \quad \text{where } k =  \sum\limits_{j=1}^n k_j \delta_j .$$
By Corollary \ref{solann} there exists $m \in \N$ such that 
$\Omega^{(m,h)}_{p,q}$
belongs to the annihilator of $M$ for all $p, q, h \in \Gamma_\mu$.

We claim that $\hM_{\beta+s}$ is equal to
\begin{equation}
\Span \left\{ \psi(e_{s-k}, M_{\beta+k}) \; \big| \; \| k \| \leq \frac{nm}{2} \right\} .
\label{span}
\end{equation}
To prove this claim we need to show that for all $p\in \Gamma_\mu$ and for all $u \in M_{\beta+p}$,
$\psi(e_{s-p}, u)$ belongs to \eqref{span}. We prove this by induction on $\| p \|$. If $|p_j| \leq \frac{m}{2}$
for all $j = 1, \ldots, n$, the claim holds trivially. Suppose that for some $j$ we have $|p_j| > \frac{m}{2}$.
Let us assume $p_j < - \frac{m}{2}$, the case $p_j >  \frac{m}{2}$ is treated in the same way.
We have that $p + \delta_j, p+ 2\delta_j, \ldots, p+m\delta_j$ all have norms less than $\| p \|$.
Since $e_0$ acts on $M_{\beta+p}$ with a non-zero scalar, we can write $u = e_0 v$ for some $v \in M_{\beta+p}$.
Let us verify that 
$$\sum_{i=0}^m (-1)^i {m \choose i} \psi(e_{s-p-i\delta_j},  e_{i\delta_j} v) = 0$$
in $\hM$. Indeed,
\begin{align*}
\sum_{i=0}^m (-1)^i {m \choose i} &\psi(e_{s-p-i\delta_j},  e_{i\delta_j} v) (t^r) 
= \sum_{i=0}^m (-1)^i {m \choose i} e_{s+\mu\cdot r-p-i\delta_j} e_{i\delta_j} v \\
&= \Omega^{(m, \delta_j)}_{s+\mu\cdot r -p, 0} v = 0.
\end{align*}
Thus
\begin{equation}
\psi(e_{s-p}, u) = - \sum_{i=1}^m (-1)^i {m \choose i} \psi(e_{s-p-i\delta_j},  e_{i\delta_j} v). 
\label{expr}
\end{equation}
By induction assumption the right hand side of \eqref{expr} belongs to \eqref{span}, and so does $\psi(e_{s-p}, u)$,
which proves the claim. Hence
$$\hM_{\beta+s} =\Span \left\{ \psi(e_{s-k}, M_{\beta+k}) \; \big| \; \| k \| \leq \frac{nm}{2} \right\} $$
and it is finite-dimensional. The theorem is proved.
\end{proof}

Using this result for the solenoidal Lie algebras we can easily establish its analogue for the Lie algebra of vector fields
on a torus.

\begin{thm}
Let $M$ be a cuspidal module for the Lie algebra $W_n$ of vector fields on a torus. Then its $A$-cover $\hM$ is cuspidal.
\label{cusp}
\end{thm}
\begin{proof}
Fix a basis $\{ \mu_1, \ldots, \mu_n \}$ of $\C^n$ such that every $\mu_j$ is generic.
Then as a vector space we can decompose $W_n$ into a direct sum of its subalgebras:
$$W_n = W_{\mu_1} \oplus \ldots \oplus W_{\mu_n} .$$
The module $M$ when restricted to each solenoidal Lie algebra $W_{\mu_j}$ remains cuspidal.
Let $\hM$ be the $A$-cover of the $W_n$-module $M$. The weight space $\hM_\lambda$ is spanned
by the set
$$\left\{ \psi(t^k d_{\mu_j} ,  M_{\lambda-k}) \; \big| \; k\in\Z^n, \,  j = 1, \ldots, n \right\} .$$
By Theorem \ref{solcusp} a subspace spanned by
$$\left\{ \psi(t^k d_{\mu_j} , M_{\lambda-k}) \; \big| \; k\in\Z^n \right\} $$
is finite-dimensional for each $j = 1, \ldots, n$. Thus $\hM_\lambda$ is finite-dimensional as well. Since
$\hM$ is an $A$-module, it is cuspidal.
\end{proof}

As a corollary of Theorems \ref{solcusp} and \ref{cusp} we obtain Theorem \ref{proj} by applying Proposition
\ref{Astr} (3). 
\ifnum \value{version}=\value{long} 
We conjecture that in fact the condition $M = \LL M$ in Theorem \ref{proj}
is superfluous.

We have thus established a close relation between cuspidal $W_n$-modules and cuspidal $AW_n$-modules.
The general structure of cuspidal $AW_n$-modules is well-understood. These modules are parametrized by finite-dimensional 
representations of a certain infinite-dimensional Lie algebra \cite{Jet}. Let us recall that construction.
Let $\J = \Der \C[ t_1, \ldots, t_n]$, which is the {\it jet} Lie algebra of $W_n$ (see \cite{BIM} for the definition
of a jet of a Lie algebra). The Lie algebra $\J$ has a natural $\Z$-grading
$$\J = \mathop\oplus\limits_{k=-1}^\infty \J_{k} ,$$
where $\J_{-1}$ is spanned by $\frac{\partial}{\partial t_1}, \ldots, \frac{\partial}{\partial t_n}$ and
$\J_0 \cong \gl_n$. Set 
$$\J_+ = \mathop\oplus\limits_{k=0}^\infty \J_{k} .$$
If $(V, \rho)$ is a finite-dimensional representation of $\J_+$ then 
\begin{equation}
\rho(\J_k) = 0 \quad \text{for } k\gg 0.
\label{locnil}
\end{equation}
\begin{thm} (\cite{Jet})
Let $M$ be a cuspidal $AW_n$-module with the set of weights $\beta + \Z^n$ for some $\beta \in \C^n$.
Then there exists a finite-dimensional representation $(V, \rho)$ of $\J_+$ such that
$$M \cong t^\beta \C [t_1^{\pm 1}, \ldots, t_n^{\pm 1} ] \otimes V$$
with the natural action of $A$ and the following $W_n$-action:
\begin{equation}
(t^m d_j) (t^s \otimes v) = s_j t^{s+m} \otimes v 
+ \sum_{k \in \Z_+^n \backslash \{ 0 \} } \frac{m^k}{k!}  t^{s+m} \otimes \rho\left(t^k \frac{\partial}{\partial t_j}\right) v ,
\label{AWcusp}
\end{equation}
where $m \in \Z^n$, $s \in \beta + \Z^n$, 
 $m^k = m_1^{k_1} \ldots m_n^{k_n}$, $k! = k_1 ! \ldots k_n !$ and $v \in V$.
\label{jets}
\end{thm}
 
In particular, one can easily obtain from Theorem \ref{jets} a classification of simple cuspidal $AW_n$-modules (\cite{Rao2}, \cite{Jet}),
which we will present in the next section.

\begin{rem} Note that \eqref{AWcusp} together with \eqref{locnil}
implies that a cuspidal $AW_n$-module always has polynomial structure constants, i.e., it
is a module with a polynomial action.
\end{rem}

We conclude this section with the discussion of duality in the category of cuspidal modules.

Define the graded dual of a weight module $M$ as
$$M^* = \mathop\oplus\limits_{\lambda} (M_\lambda)^* .$$
It is clear that $M^*$ is again a weight module.
\begin{lem}
If $M$ is an $A\LL$-module then its graded dual $M^*$ is also an $A\LL$-module.
\end{lem}
Verification of this lemma is straightforward and we omit its proof.

\begin{prp}
Let $M$ be a weight $\LL$-module. Then the map
$$ \pi^*: \; M \rightarrow \left( \widehat{M^*} \right)^*$$
given by
$$ \pi^* (u) (\psi(x, \xi)) = \xi (x u), \quad \text{for } u\in M, \; x \in\LL, \; \xi \in M^*, $$
is a homomorphism of $\LL$-modules with kernel $\{ u\in M \,|\, \LL u = 0 \}$.
\end{prp}
\begin{proof}
Let us show that $\pi^*$ is a homomorphism of $\LL$-modules. We need to show that
for $y \in \LL$ we have $y \pi^* (u) = \pi^* (yu)$. Indeed,
\begin{align*}
(y \pi^* (u)) ( \psi(x,\xi)) = - \pi^*(u) (y \psi(x, \xi))
= -\pi^* (u) (\psi([y,x], \xi) - \pi^*(u) (\psi(x, y\xi)) \\
= - \xi ([y,x] u) - (y \xi) (xu) = - \xi ([y,x] u) + \xi( y (x u)) 
= \xi ( x (yu)) = \pi^* (yu) (\psi(x, \xi)) .
\end{align*}
The statement about the kernel of $\pi^*$ is obvious.
\end{proof}

\begin{cor}
Let $M$ be a cuspidal $\LL$-module with no 1-dimensional submodules. Then 
$M$ can be embedded into a cuspidal $AW_n$-module.
\end{cor}
\fi

\section{Classification of simple $W_n$-modules with finite-dimensional weight spaces}

In this section we establish the classification of simple $W_n$-modules with finite-dimensional weight
spaces which was originally conjectured by Eswara Rao \cite{Rao2}. First we are going to present two classes
of modules: cuspidal simple modules and simple modules of the highest weight type. Then we shall
prove that these two classes give a complete list.

{\it Cuspidal simple modules.} These modules have a geometric origin -- they are modules of tensor 
fields on a torus (\cite{Ru}, \cite{Fu}, \cite{Sh}). 

Let $U$ be a finite-dimensional simple $\gl_n$-module and let $\beta\in\C^n$. The module of tensor 
fields $T(U,\beta)$ is the space
$$T(U, \beta) = t^\beta \C [t_1^{\pm 1}, \ldots, t_n^{\pm 1}] \otimes U$$
with a $W_n$-action given as follows:
$$(t^m d_a) (t^s \otimes u) = s_a t^{s+m} \otimes u + \sum_{p=1}^n m_p t^{s+m} \otimes E_{pa} u,$$
where $m \in \Z^n$, $s \in \beta + \Z^n$, $u \in U$ and $E_{pa}$ is an $n \times n$ matrix with $1$ in position
$(p,a)$ and zeros elsewhere. It is easy to check that the modules of tensor fields $T(U,\beta)$ are $AW_n$-modules.

Let $V$ be the natural $n$-dimensional representation of $\gl_n$ and let $\Lambda^k V$ be its $k$-th exterior power,
$k = 0, \ldots, n$. Note that $\Lambda^k V$ is a simple $\gl_n$-module for all $k = 0, \ldots, n$. 
The modules $\Omega^k (\beta) = T(\Lambda^k V, \beta)$ are the modules of differential
$k$-forms. These modules form the de Rham complex
$$\Omega^0(\beta) \darrow \Omega^1 (\beta) \darrow \ldots \darrow \Omega^n (\beta) .$$
The differential $d$ of the de Rham complex is a homomorphism of $W_n$-modules (however it is not a homomorphism
of $A$-modules). Thus the kernels and the images of $d$ are $W_n$-submodules in $\Omega^k (\beta)$. As a result,
for $1 \leq k \leq n-1$, $\Omega^k (\beta)$ are reducible $W_n$-modules, while $\Omega^0 (\beta)$ and $\Omega^n (\beta)$
are reducible if and only if $\beta \in \Z^n$.

The following Theorem was proved by Eswara Rao (there is also a similar result of Rudakov \cite{Ru} for the Lie algebra
of vector fields on an affine space):
\begin{thm} (\cite{Rao1}, see also \cite{GZ}) Let $U$ be a finite-dimensional simple $\gl_n$-module and let $\beta\in\C^n$.

(1) The module of tensor fields $T(U, \beta)$ is a simple $W_n$-module unless it is a module $\Omega^k (\beta)$ of
differential $k$-forms, $0 \leq k \leq n$.

(2) The module $\Omega^k (\beta)$, $0 \leq k \leq n$, has a unique simple quotient, which is $d \Omega^k(\beta)$ for $0 \leq k \leq n-1$.
The module $\Omega^n (\beta)$ is simple when $\beta \not\in\Z^n$ and has a trivial $1$-dimensional module as a simple quotient
when $\beta \in \Z^n$. 
\label{deRham}
\end{thm}

Eswara Rao also classified simple cuspidal $AW_n$-modules \cite{Rao2}. Cuspidal $AW_n$-modules without the assumption on simplicity 
are described in \cite{Jet}.

\begin{thm} (\cite{Rao2})
Simple cuspidal $AW_n$-modules are precisely the modules of tensor fields $T(U, \beta)$ with simple finite-dimensional $gl_n$-modules $U$
and $\beta\in\C^n$.
\label{AWsimple}
\end{thm}

{\it Simple $W_n$-modules of the highest weight type.} Consider a $\Z$-grading on $W_n$ by eigenvalues of $d_n = t_n \frac{\partial}{\partial t_n}$.
Write the decomposition of $W_n$:
$$W_n = W_n^- \oplus W_n^0 \oplus W_n^+$$
into the subalgebras of negative, zero and positive degree. The subalgebra $W_n^0$ is a semidirect product of $W_{n-1}$ with an abelian ideal:
$$W_n^0 = \Der \C [t_1^{\pm 1}, \ldots, t_{n-1}^{\pm 1} ] \ltimes \C [t_1^{\pm 1}, \ldots, t_{n-1}^{\pm 1} ] d_n .$$
Let $\tX$ be a weight module for $W_n^0$ with a unique simple quotient $X$ (we could have started with $X$, but keeping in mind the construction
of simple modules $d \Omega^k (\beta)$, we prefer to work in this slightly more general setting).
We define simple modules of the highest weight type using the technique of the generalized Verma modules.

Postulate $W_n^+ \tX = 0$ and construct the generalized Verma module by inducing:
$$M(\tX) = \Ind_{W_n^0\oplus W_n^+}^{W_n} \tX \cong U(W_n^-) \otimes \tX .$$
The generalized Verma module $M(\tX)$ has a unique simple quotient, which we denote $L(\tX)$. Note that $L(\tX) \cong L(X)$.

The group $\GL$ acts on the algebra $A$ of Laurent polynomials, which is the group algebra of $\Z^n$. This induces
the action of $\GL$ by automorphisms on $W_n = \Der A$. For each $g \in \GL$ we can consider a twisted module $L(\tX)^g$, where the action
of $W_n$ is composed with an automorphism $g$.

 We call the modules $L(\tX)^g$ simple $W_n$-modules of the {\it highest weight type}. Whereas the generalized Verma module $M(\tX)$ always
contains infinite-dimensional weight spaces, the situation may improve with passing to its simple quotient $L(\tX)$.
\begin{thm} (\cite{BB}, see also \cite{BZ})
Let $\tX$ be a $W_n^0$-module with a polynomial action (see \cite{BB} for the definition). Then the simple $W_n$-module of the highest weight type
$L(\tX)$ has finite-dimensional weight spaces.
\end{thm}
In particular, we can take a simple finite-dimensional $\gl_{n-1}$-module $U$, and let $\tX$ be a $W_{n-1}$-module of tensor fields
$T(U, \beta)$ for some $\beta\in\C^{n-1}$. We fix $\gamma\in\C$ and define the action of the abelian ideal as follows:
$$(t^m d_n)(t^s \otimes u) = \gamma t^{s+m} \otimes u, \quad \text{ for } m\in\Z^{n-1}, \; s\in\beta+\Z^{n-1}, \; u \in U.$$
The resulting module $\tX = T(U, \beta, \gamma)$ is a $W_n^0$-module with a polynomial action \cite{BZ} and has a unique simple quotient.
Thus the corresponding simple $W_n$-module of the highest weight type $L(\tX) = L(U, \beta, \gamma)$ has finite-dimensional weight spaces.

The structure of the simple modules $L(U, \beta, \gamma)$ was studied in \cite{BF}, where their vertex operator realizations are given.

Now we can state the main classification result presented in Introduction.

\begin{thm}
(1) Every simple $W_n$-module with finite-dimensional weight spaces is either cuspidal or of the highest weight type.

(2) A simple cuspidal $W_n$-module is isomorphic to one of the following:

$\bullet$ a module of tensor fields $T(U, \beta)$, where $\beta\in\C^n$ and $U$ is a finite-dimensional simple $\gl_n$-module different
from the $k$-th exterior power of the natural $n$-dimensional $\gl_n$-module, $0 \leq k \leq n$;

$\bullet$ a submodule $d \Omega^k (\beta) \subset \Omega^{k+1} (\beta)$, $0 \leq k < n$;

$\bullet$ a trivial $1$-dimensional $W_n$-module.

(3) A module of the highest weight type is isomorphic to $L(U, \beta, \gamma)^g$ for some finite-dimensional simple
$\gl_{n-1}$-module $U$, $\beta\in\C^{n-1}$, $\gamma\in\C$ and $g \in \GL$.
\label{main}
\end{thm}

\begin{proof}
Part (1) of the Theorem was proved by Mazorchuk-Zhao (\cite{MZ}, Theorem 1). More precisely, they prove
that a simple $W_n$-module with finite-dimensional weight spaces is either cuspidal or is isomorphic to a module
of the highest weight type $L(X)^g$ for some $g\in\GL$ and some simple cuspidal $W_n^0$-module $X$. Let us analyze these
two cases.

{\it Case of cuspidal modules.} The classification in this case follows from the following lemma:
\begin{lem}
Let $M$ be a simple cuspidal $W_n$-module. Then $M$ is a simple quotient of a module of tensor fields $T(U, \beta)$
for some simple $\gl_n$-module $U$ and $\beta\in\C^n$.
\label{cuspquot}
\end{lem}
\noindent
{\it Proof of Lemma.}
The claim holds for the trivial $1$-dimensional module since it is the simple quotient of the module $\Omega^n (0)$.
Let us now assume that $M$ is a non-trivial simple module. Then $W_n M = M$ and
there is a surjective homomorphism $\pi$ from a cuspidal $AW_n$-module $\hM$ to $M$
by Proposition \ref{Astr}(3) and Theorem \ref{cusp}.

Consider a composition series of $AW_n$-submodules in $\hM$:
$$0 = \hM_0 \subset \hM_1 \subset \ldots \subset \hM_{\ell-1} \subset \hM_\ell = \hM$$
with the quotients $\hM_i / \hM_{i-1}$ being simple $AW_n$-modules. Let $s$ be the smallest integer
such that $\pi(\hM_s) \neq 0$. Since $M$ is a simple $W_n$-module we have
$\pi(\hM_s) = M$ and $\pi(\hM_{s-1}) = 0$. This gives us a surjective homomorphism of $W_n$-modules
$$\bar{\pi}: \; \hM_s /\hM_{s-1} \rightarrow M .$$
By Theorem \ref{AWsimple}, $\hM_s /\hM_{s-1}$ is isomorphic to a module of tensor fields $T(U, \beta)$
for some simple $\gl_n$-module $U$ and $\beta\in\C^n$. This completes the proof of the lemma.
\hfill \qed

The classification of simple cuspidal $W_n$-modules now follows from Lemma \ref{cuspquot} and Theorem \ref{deRham}.

{\it Case of simple modules of the highest weight type.} Let $M$ be a simple module $L(X)^g$ for some
$g\in\GL$ and some simple cuspidal $W_n^0$-module $X$. Guo-Liu-Zhao \cite{GLZ} have reduced the
description of these modules to the classification of simple cuspidal $W_{n-1}$-modules. More precisely,
they proved (\cite{GLZ}, Theorem 3.3) that either $X \cong T(U, \beta, \gamma)$ for some
finite-dimensional simple $\gl_{n-1}$-module $U$, $\beta \in \C^{n-1}$ and some non-zero $\gamma \in \C$,
or the abelian ideal $\C[t_1^{\pm 1}, \ldots, t_{n-1}^{\pm 1}] d_n$ acts trivially on $X$. The latter case is
the only one that we need to consider. In this setting $X$ is just a simple cuspidal $W_{n-1}$-module.
By Lemma \ref{cuspquot}, $X$ is a simple quotient of a module $\tX = T(U, \beta)$  for some
finite-dimensional simple $\gl_{n-1}$-module $U$ and $\beta \in \C^{n-1}$. Then as a $W_n^0$-module,
$\tX$ is isomorphic to $T(U, \beta, 0)$ and $M \cong L(U, \beta, 0)^g$. This completes the proof of the theorem.
\end{proof}

\

\section{Acknowledgements}
The first author is supported in part by a grant from the Natural Sciences
and Engineering Research Council of Canada.  
The second author is supported in part by the CNPq grant
(301743/2007-0) and by the Fapesp grant (2010/50347-9).
Part of this work was carried out during the visit of the first author
to the University of S\~ao Paulo in 2013. This author would like
to thank the University of S\~ao Paulo for hospitality and
excellent working conditions and Fapesp (2012/14961-0) for
financial support.


\begin{thebibliography}{9999}

\bibitem{BB} S.~Berman, Y.~Billig,
{Irreducible representations for toroidal Lie algebras}, 
J. Algebra {\bf 221} (1999), 188-231.

\bibitem{Jet} Y.~Billig,
{Jet modules},
Canad.J.Math, {\bf 59} (2007), 712-729.

\ifnum \value{version}=\value{long} 

\bibitem{BIM} Y.~Billig, K.~Iohara, O.~Mathieu,
{Bounded representations of the determinant Lie algebra},
in preparation.
\fi

\bibitem{BF} Y.~Billig, V.~Futorny,
{Representations of Lie algebra of vector fields on a torus and chiral de Rham complex},
to appear in Trans. Amer. Math. Soc.

\bibitem{BZ} Y. ~Billig, K.~Zhao,
{Weight modules over exp-polynomial Lie algebras},
J. Pure Appl. Algebra {\bf 191} (2004), 23-42.

\bibitem{Con} C.~H.~Conley,
{Bounded length 3 representations of the Virasoro Lie algebra},
Internat. Math. Res. Notices, No.12 (2001), 609-628.

\bibitem{ConMar} C.~H.~Conley, C.~Martin,
{Annihilators of tensor density modules}, 
J. Algebra {\bf 312} (2007), 495-526.

\bibitem{Rao1} S.~Eswara Rao,
{Irreducible representations of the Lie algebra of the
diffeomorphisms of a $d$-dimensional torus}, J. Algebra {\bf 182}
(1996), 401-421.

\bibitem{Rao2} S.~Eswara Rao,
{Partial classification of modules for Lie algebra of
diffeomorphisms of $d$-dimensional torus}, J. Math. Phys. {\bf 45}
(2004), no. 8, 3322-3333.

\bibitem{FF} B.~L.~Feigin, D.~B.~Fuks, 
{Homology of the Lie algebra of vector fields on the line}, 
Funct. Anal. Appl. {\bf 14} (1980), 201-212.

\bibitem{Fu} D.~B.~Fuks,
{Cohomology of infinite-dimensional Lie algebras},
Contemporary Soviet Mathematics. Consultants Bureau, New York, 1986. 

\bibitem{GLZ} X.~Guo, G.~Liu, K.~Zhao,
{Irreducible Harish-Chandra modules over extended Witt algebras},
to appear in Ark. Mat.

\bibitem{GZ} X.~Guo, K.~Zhao,
{Irreducible weight modules over Witt algebras},
Proc. Amer. Math. Soc. {\bf 139} (2011), 2367-2373.

\bibitem{Kac} V.~G.~Kac,
{Some problems of infinite-dimensional Lie algebras and their representations},
Lecture Notes in Mathematics, {\bf 933} (1982), 117-126. Berlin, Heidelberg, New York: Springer.

\bibitem{Kap} I.~Kaplansky,  
{The Virasoro algebra},
Comm. Math. Phys. {\bf 86} (1982), 49-54.

\bibitem{KS} I.~Kaplansky,  L.~J.~Santharoubane,  
{Harish-Chandra modules over the Virasoro algebra},
MSRI Publications vol. 5, pp. 217-231. Berlin, Heidelberg, New York: Springer 1985.

\bibitem{LZ} R.~Lu, K.~Zhao, 
{Classification of irreducible weight modules over higher rank Virasoro algebras},
Adv. Math. {\bf 206} (2006), 630-656.

\bibitem{MarPia1} C.~Martin, A.~Piard,
{Indecomposable modules over the Virasoro Lie algebra and a conjecture of V.~Kac},
Comm. Math. Phys. {\bf 137} (1991), 109-132.

\bibitem{MarPia2} C.~Martin, A.~Piard,
{Classification of the indecomposable bounded admissible modules over the Virasoro
Lie algebra with weightspaces of dimension not exceeding two},
Comm. Math. Phys. {\bf 150} (1992), 465-493. 

\bibitem{Mat} O.~Mathieu,
{Classification of Harish-Chandra modules over the Virasoro
algebra}, Invent. Math. {\bf 107} (1992), 225-234.

\bibitem{MZ} V.~Mazorchuk, K.~Zhao,
{Supports of weight modules over Witt algebras}, Proc. Royal Soc. Edinburgh: Section A
{\bf 141} (2011), 155-170.

\bibitem{Ru} A.~N.~Rudakov,
{Irreducible representations of infinite-dimensional Lie algebras
of Cartan type}, Izv. Akad. Nauk SSSR Ser. Mat. {\bf 38} (1974),
835-866.

\bibitem{Sh} G.~Shen, 
{Graded modules of graded Lie algebras of Cartan type. I. Mixed products of modules}, 
Sci. Sinica Ser. A {\bf 29} (1986), 570-581.

\bibitem{Su} Y.~Su,
{Simple modules over the high rank Virasoro algebras},
Comm. Algebra, {\bf 29} (2001), 2067-2080.

\end{thebibliography}
\end{document}